\documentclass[12pt,leqno]{amsart}

\usepackage[active]{srcltx}
\usepackage{amsmath} 
\usepackage{amsfonts}   
\usepackage{amssymb} 
\usepackage{amsthm} 
\usepackage{latexsym}
\usepackage{a4wide}  
\usepackage{amsmath}  
\usepackage{amsfonts}   
\usepackage{amssymb}
\usepackage{amsthm} 
\usepackage{latexsym}  
\usepackage{color}

 \sloppypar \textheight 21cm \textwidth 15cm
\newtheorem{theorem}{Theorem}[section]
\newtheorem{definition}[theorem]{Definition}  
\newtheorem{lemma}[theorem]{Lemma}
\newtheorem{remark}[theorem]{Remark}
\newtheorem{proposition}[theorem]{Proposition}
\newtheorem{cor}[theorem]{Corollary}
\newtheorem{corollary}[theorem]{Corollary}

\newcommand{\supp}{\text{supp}}

\renewcommand{\H}{{\mathcal H}}
\def\C{\mathbb C}
\def\R{\mathbb R}
\def\S{{\mathcal S} }

\def\C{\mathbb C}
\def\R{\mathbb R}
\def\I{\mathbb I}
\def\N{\mathbb N}
\def\al{\alpha}

\def\de{\delta}
\def\rh{\rho}
\def\ga{\gamma}

\def\va{\varepsilon}
\def\LA{\Lambda}
\def\la{\lambda}

\def\va{\varphi}
\def\sp#1#2{\langle{#1},{#2}\rangle}

\def\l{\mathfrak{l}}

\def\ga{\gamma}
\def\la{\lambda}
\def\ve{\varepsilon}

\def\ch{\chi}
\def\et{\eta}
\def\ps{\psi}

\def\N{\mathbb{N}}

\def\R{\mathbb{R}}
\def\C{\mathbb{C}}

\def\ol#1{\overline{#1}}

\def\R{{\mathbb R}}
\def\C{{\mathbb C}}
\def\N{{\mathbb N}}

\def\I{{\mathbb I}}

\def\B{{\mathcal B}}

\def\F{{\mathcal F}}

\def\H{{\mathcal H}}

\def\L{{\mathcal L}}

\def\iy{\infty}

\def\ol#1{\overline{#1}}

\def\no#1{\Vert #1\Vert }

\def\noag#1{\Vert #1\Vert_{\mathrm A(G)} }

\def\CC#1{ C_c(#1)}

\def\supp#1{\text{supp}(#1)}

\def\inv{^{-1}}

\def\l2#1{L^2(#1)}
\def\L1#1{L^1(#1)}

\def\limk{\lim_{k\to\infty}}

\def\nn{\nonumber}

\def\limi{\lim_{i\to\infty}}
\def\sp#1#2{\langle #1,#2\rangle }%skalarprodukt

\def\L#1#2{L^{#1}(#2)}

\def\Re{\mathrm{\, Re\, }}

\def\lef({\left(}
\def\rig){\right)}

\newcommand\nul{\mathop{\rm null}}

\begin{document}

\title[OSpetral Sets]{On the structure of Spectral Sets}
\author[J.Ludwig]{Jean Ludwig}
\address{ Universit\'e de Lorraine\\
Institut Elie Cartan de Lorraine\\
UMR 7502, Metz, F-57045, France.}
\email{jean.ludwig@univ-lorraine.fr.}

\author[l.Turowska]{Lyudmila  Turowska}
\address{ Department of Mathematics, Chalmers University of
Technology and University of Gothenburg \\
 SE-412 96 G\"oteborg, Sweden}
 \email{turowska@chalmers.se}
\thanks{2020 Mathematics Subject Classification: Primary 43A45, Secondary 43A46\\Keywords: locally compact group, Fourier algebra, spectral synthesis, Ditkin set}
\date{}

\begin{abstract} We discuss convergence in the Fourier algebra $A(G)$ of a locally compact group $G$ and obtain a new characterization of local spectral sets of $G$.
\end{abstract}

  \maketitle{}

\section{Introduction}

The notion of  spectral synthesis was introduced by A.Beurling in the late 
1930th. Since then it has been a subject of extensive research in Harmonic Analysis 
primarily in the context of the Fourier algebra $A(G)$ for a locally compact group $G$. 
When $G$ is abelian then $A(G)$ is isometrically isomorphic to $L^1(\widehat 
G)$, the $L^1$-algebra  of the dual group $\widehat G$. If $G$ is 
non-commutative, the Fourier algebra is defined directly on $G$ as the algebra of 
matrix coefficients of the left regular representation of $G$.
The algebra $A(G)$ is a semisimple, regular, commutative Banach algebra with the 
Gelfand spectrum $G$: characters correspond to evaluation at points of $G$.

The question 
of spectral synthesis can be framed as a question  about the ideals of $A(G)$, 
specifically whether a given closed ideal $I\subset A(G)$ is the intersection of all 
maximal ideals containing $I$. This intersection can be 
identified with a closed subset of $G$, known as  the hull of $I$. A closed subset $E$ 
of $G$  is said to be a set of spectral synthesis if the only closed ideal having $E$ as 
its hull is the intersection of all maximal ideals associated with the 
points of $E$.

The first example demonstrating the failure of spectral synthesis was provided by 
L. Schwartz in 1948 for $A(\mathbb R^n)$, $n\geq 3$. That synthesis fails for any 
$A(G)$, where $G$ is a non-discrete locally compact abelian group, was proved by 
P. Malliavin  in 1959. This result was later extended to arbitrary locally compact groups, under mild additional assumption, by E. Kaniuth and A.T.M. Lau, see e.g. \cite[Theorem 6.2.3]{kaniuth-lau}.
To classify sets of spectral synthesis seems to be a problem out of reach for 
the moment. Only some special classes of sets of spectral synthesis  have been 
identified (see, e.g. \cite{graham}).
 An outstanding problem in the field,  as noted in \cite{graham}, 
is  whether the union of two sets of spectral synthesis is itself a 
set of spectral synthesis. The question  was initially posed by C. Herz   in  \cite{H1} for abelian groups and reiterated three  years later  by H. Reiter in  
\cite{Re}. To date, no solution has been found, underscoring the need for new tools to analyze sets of spectral synthesis and advance our understanding of the union problem. 
%We note that the answer to this question is negative for general commutative 
%semisimple regular Banach algebras as it was shown by A. Atzmon in 
%\cite{atzmon}. 

In this paper we present a  characterisation of sets of spectral  synthesis
formulated as a Hilbert space approximation. 
%Our hope is that it may give a new insight into the union problem. 
The paper is organised as follows. In Section 2 we recall definitions and fix 
notations. In Section 3 convergence properties in $A(G)$ are discussed.   Theorem 
\ref{l2l2} and Corollary \ref{ckak} provide  conditions for a sequence to converge in $A(G)$ which maybe of independent  interest.  These results are used to establish a characterization of local spectral sets in Theorem \ref{spectralstrcon}. In Section \ref{abelian} a refinement of  Theorem \ref{spectralstrcon} is presented for  the abelian case. 
The notion of strong spectral sets is introduced in Section \ref{strong}, and it is shown that the union of two such sets remains strongly spectral. This class includes Ditkin sets. 
%and in 
%Section \ref{abelian} we show that in the abelian case every spectral set  is 
%strongly spectral, solving in this way the union problem. 
%In Section 6 we introduce the notion of strongly spectral sets and show that the union of two such sets is again strongly spectral. The class of strongly spectral sets includes the class of Ditkin sets. 
  Finally in Appendix some formulas related to the action of the von Neumann 
algebra of $G$ on the corresponding Fourier algebra are collected.

\section{Preliminaries and Notations}

Let $G$ be a locally compact group with fixed left Haar measure $m$. We denote 
by $C(G)$ the set of continuous complex-valued functions on $G$ and by $C_c(G)$, 
those with compact support in $C(G)$.  For any $1\leq p\leq\infty$ we let 
$L^p(G)$ denote the usual $L^p$ space with respect to $m$ with norm 
$\|\cdot\|_p$.
For any appropriate pair of complex-valued functions $f$, $g$, we denote 
$(f*g)(t)=\int_G f(s)g(s^{-1}t)ds$. We set as
customary $\check{\xi}(s) = \xi(s^{-1})$ and $\tilde \xi(s)=\ol{\xi(s\inv)} $, 
$s\in G$.

The \emph{Fourier algebra}
$A(G)$ of $G$  was defined by P. Eymard in \cite{E}. We recall that $A(G)$  is the 
algebra of coefficients of the left regular representation $\lambda:G\to 
B(L^2(G))$, $(\lambda(s)f)(t)=f(s^{-1}t)$,
that is, the algebra of
functions of the form $s\mapsto \langle\lambda(s)f, g\rangle=\bar g*\check f(s)$, for
$f,g\in L^2(G)$. $A(G)$ is a semi-simple, regular, commutative Banach algebra of continuous 
functions on $G$ with respect to the norm
$$\|u\|_{A(G)}=\text{inf}\{\|f\|_2\|g\|_2:u=\bar g*\check f\}.$$
The Gelfand spectrum of $A(G)$ is known to be the space $G$ itself.

We also recall the duality relation $A(G)^*\simeq VN(G)$, where $VN(G)$ is the 
von Neumann algebra generated by $\lambda(G)$.
The duality is given by the pairing $\langle u,T\rangle=\langle Tf,g\rangle$, where $u\in 
A(G)$,
$u(s) = \langle\lambda(s)f,g\rangle$. We use the same notation, $\langle\cdot,\cdot\rangle$, both  for the inner product on a Hilbert space and for the duality pairing; which one is used should be clear from the context.

If $I\subseteq A(G)$ is an ideal, we define the hull of $I$ to be the set
$$ h(I) = \{s\in G : u(s) = 0 \mbox{ for all } u\in I\}\subseteq G.$$
On the other hand, for a closed set  $E\subseteq G$, define the kernel $k(E) $ 
of $E $ and the minimal ideal $j(E) $ with the hull $E $ by 
\begin{eqnarray*}
 \nn k(E)&:= &
 \{f\in A(G) : f(s)=0, s\in E\},\\ 
 \nn  j(E)
&:= &
\{f\in A(G) : f \text{ has compact support disjoint from }
E\}.
 \end{eqnarray*}

We have that
$h(\overline{j(E)}) =  h(k(E)) = E$ and that if
$I\subseteq A(G)$ is a closed ideal with $h(I) = E$, then
$\overline{j(E)}\subseteq I\subseteq k(E)$.

A closed subset $E\subseteq G$ is said to be {\it spectral}  or  a {\it set of 
spectral synthesis} if $k(E)=\overline{j(E)}$, equivalently
 $k(E)$ is the only closed ideal whose hull is the set $E$;  $E$ is called a 
{\it Ditkin set} if $a\in \overline{aj(E)}$ for any $a\in k(E)$. Clearly any 
Ditkin set is a set of spectral synthesis. The converse is an open problem.

We say that $E$ is  \textit{local spectral} or of \textit{local 
spectral synthesis} if for any
$u\in A(G)\cap C_c(G)$, which  vanishes on $E$, there
exists a sequence $(u_n)_{n\in\N}\subset A(G)$  which converges to
$u$ and such that $u_n$ vanishes on  a neighbourhood of $E$ for
every $n\in \N$.

It is well known that every local spectral subset of $G $ is spectral, provided 
that $A(G) $ has an (unbounded) approximate identity, for instance if $G$ is 
amenable, weakly amenable (\cite{coh}) or, more generally,  if every $u\in A(G) $ 
is contained in the closed ideal it generates.  Many examples of local spectral 
subsets are known. For example it was shown by M. Takesaki and N. Tatsuuma in \cite{takesaki-tatsuuma} for $A(G)$ and by C. Herz in \cite{H2} for the Fig\`a-Talamanca-Herz algebras $A_p(G)$ that 
closed subgroups of any locally compact group are local spectral. 

We finish this section with setting basic notation and 
introducing some notions that will be used subsequently. We write $B(H)$ for the algebra of all bounded operators on a Hilbert space $H$.  For an (unbounded) operator $A$ on $H$, $\mathcal D(A)\subset H$ will stand for the domain  of $A$. For $W\subset G$ we write $W^c$ for its complement. For  a function $f:G\to \C $ we let  $\text{null}( f)=\{x\in G, f(x)=0\} $ and $\text{supp}(f)=\overline{\text{null}(f)^c}$.

\section{A characterization of local spectral sets}
In this section we will present a new characterisation of local spectral sets. 
But first we examine some convergence properties in $A(G)$.

\subsection{Convergence in $A(G)$}\label{conv AG}

We recall first the right and the left  action of $VN(G)$ on $A(G)$:   if $u\in 
A(G)$, $T\in VN(G)$, the elements $u\cdot T$ and $T\cdot u$ in $A(G)$ are defined  through the 
following formulas:
$$\langle S, u\cdot T\rangle:=\langle TS, u\rangle,\quad \langle S, T\cdot u\rangle:=\langle ST, u 
\rangle, \quad \forall S\in VN(G).$$
{ If $u\in A(G)\cap L^2(G)$, we shall also write $T(u)$ for the action 
of $T$ on $u\in L^2(G)$.

Recall that a continuous function $u:G\to\mathbb C$ is called positive definite 
if for each $n\in\mathbb N$ and  $s_1$, $s_2,\ldots, s_n$ in $G$ the matrix 
$(u(s_i^{-1}s_j))_{j,j=1}^n$ is positive definite.

We say that a locally integrable function $\varphi$ on $G$ is a function of 
positive type if 
$$\int \varphi(s)(f\ast \tilde f)(s)ds\geq 0 \text{ for all }f\in C_c(G).$$
If such $\varphi$ is continuous we obtain the previous definition of positive 
definiteness, \cite[13.4.4]{Di2}. Moreover, if $u$ is positive definite, then there exists a unitary representation $\pi$ of $G$ on $H_\pi$ and a vector $\xi\in H_\pi$ such that $u(s)=\langle \pi(s)\xi,\xi\rangle$; we write $c_\xi^\pi$ for the latter matrix coefficient.    
\smallskip 

If $f\in \CC G$ then for any $b\in L^2(G)$, $f\ast b\in L^2(G)$ and we can 
consider
the linear operator $f\mapsto f\ast
b$, $f\in \CC G,$ on $L^2(G)$. If $b$ is of positive type then
$$\langle f\ast b,f\rangle \geq 0$$
see \cite[13.7.6, 13.8.1]{Di2}.  Therefore, the operator $f\mapsto 
r(b)(f):=f\ast b$ 
defined on $\CC G\subset L^2(G)$ is positive. Let $\rho(b)$ denote its 
Friedrichs extension which is a positive selfadjoint operator. By 
\cite[13.8.3]{Di2}, if $h$ is in the domain, $\mathcal D(\rho(b))$, of 
$\rho(b)$, then $\rho(b)h=h\ast 
b$.}

\begin{lemma}\label{sqarer}\marginpar{\tiny }
  Let $b\in C(G)$ be a positive definite square integrable function. Then there 
exists a square integrable function $c$ of positive type such that 
$b=c\ast\tilde c=c\ast c$. Moreover, in this case $\rho(c)f=\rho(b)^{1/2} f$ for 
all $f\in C_c(G)$. 
  \end{lemma}
  \begin{proof}
 The existence of a square integrable function $c$ of positive type satisfying 
$b=c\ast \tilde c=c \ast c$ follows from \cite[13.8.6]{Di2}. 
 %In particular, one has  one can show that if $b$ is also 
%continuous one can choose  a positive definite function $c\in L^2(G)$ with the 
%property that  $b=c\ast c$ so that  $\rho(c)f=\sqrt{\rho(b)}f$ for any $f\in 
%C_c(G)$.
%Indeed, in \cite[13.8.6]{Di2} a positive definite function $c$ 
It is constructed as 
 the $L^2(G)$-limit of a sequence $(c_n)_{n\in\N} $ of continuous positive 
definite functions,  such that $\rho(c_n)=\sqrt{\rho(b)} E_n$, where 
$\rho(b)=\int_{[0, +\infty]} xdE(x)$ is the spectral decomposition of $\rho(b)$ and $E_n=E([0,n])$, $n\in\mathbb N$.

For $f\in \CC 
G$, one has $\rho(c_n)f=f\ast c_n$ and,  as $c_n\to c$ in $L^2(G)$, it gives
$\rho(c_n)f\to\rho(c)f$, since $\no{\rho(c_n)f-\rho(c)f}_2\leq \no 
f_1\no{c_n-c}_2 $ for any $n $.

On the other hand, $\sqrt{\rho(b)}E_nf\to \sqrt{\rho(b)}f$ for all $f\in 
{\mathcal D}(\sqrt{\rho(b)})$; 
in particular,  we have the convergence for all $f\in\CC G$, 
since $C_c(G)\subset {\mathcal D}(\rh(c))\subset {\mathcal D}(\sqrt {\rh(c)}) $. 
Therefore 
$\rho(c)f=\sqrt{\rho(b)}f$ for $f\in\CC G$.
\end{proof}
We call the function $c$  from the lemma the {\it positive square root} of $b$. 
Note that it is unique as element in $L^2(G)$. 

\smallskip 

As $A(G)$ is the predual of $VN(G)$, each $u\in A(G)$ possesses  a polar 
decomposition: there exists 
a unique pair $(A,p)$, where $p$ is a positive definite function in  $A(G)$ 
such 
that $\|u\|_{A(G)}=\|p\|_{A(G)}$, and $A$ is a partial isometry in $VN(G)$ with 
the 
final space equal to the support, $s(p)$, of $p$ and  such that
$u=A\cdot p$ and $p=A^*\cdot u$  (see \cite{Di1}, 1.4.  Th\'eor\`eme 4). The element $p$ is 
called the {\it absolute value} of $u$.

Note that for $\tilde u$ and  $\check u$ as defined above, we have $\tilde u$,  
$\check u\in A(G)$ if $u\in A(G)$ and $u=\tilde u$ if $u$ 
is positive definite.
We can define a linear involution  $T\mapsto \check T$ on $VN(G)$ by
 \begin{equation*}
 \sp {\check T} u:=\sp T{\check u},\ T\in VN(G),\ u\in A(G).
 \end{equation*}
  and  an  antilinear involution $T\mapsto\bar T$ of $VN(G)$ by
  $$\bar T(f)=\overline{T(\bar f)}, f\in L^2(G)$$
  In the Appendix we collected different equalities involving $\check T$, $\bar 
T$, 
$\check u$ and $\tilde u$.

\begin{proposition}\label{decomp}
Let $u$ be an element of
$A(G)$ such that $u$ and $\tilde u$ are in $L^2(G)$. Let $\tilde
u= A\cdot p$ be the polar decomposition of $\tilde u$. Then $p\in L^2(G)$ and  
$$p=c\ast 
c, p=\check A(u)  \text{ and }u= \bar A(p)= \bar A(c)\ast c,$$ where $c $ is  
the positive square root of $p 
$.
\end{proposition}

\begin{proof}
If $\tilde u=A\cdot p$ is the polar decomposition of $\tilde u\in A(G)\cap L^2(G)$
then by equalities (\ref{inversion}) and (\ref{l2ag}) from the Appendix we have
 \begin{equation}\label{podeco}
 p=\tilde p=(A^*\cdot \tilde u)\tilde{}=u\cdot A
=\check A(u)\in L^2(G).
\end{equation}
%where $\check A(u)$ is the action of $\check A$ on $u\in L^2(G)$.

%Since $u\in A(G)\cap L^2(G)$, by \cite[Proposition 3.17]{E}, $p=\check A(u)\in 
%A(G)\cap L^2(G)$.
Hence, by Lemma \ref{sqarer},
\begin{equation}\label{squareroot}
p=c\ast c,
\end{equation}
where $c \in L^2(G)$ is the positive square root of $p $.
 Using  (\ref{inversion}),
 (\ref{l2ag}) and (\ref{leri})  we see that 
 \begin{eqnarray}\label{udecomposed}
 \nn u &= &
 (A\cdot p)\tilde{}=p\cdot  A^*=\bar A(p)=\bar 
A(c)\ast c.
%\nn  &= &
%d\ast c, d=\bar 
%A(c).  
 \end{eqnarray}
 
\end{proof}

Replacing the partial isometry $A\in VN(G)$ by the partial isometry $\bar A\in 
VN(G)$ we call the representation $u=A(c)\ast c= A(c)\ast\tilde c$ from Proposition \ref{decomp} 
the {\it canonical representation} of $u\in A(G)\cap L^2(G)$.
%\begin{remark}\label{fastpetisfetastp}

Define for $p\in\L2G $ the subspace $\l2G_p $ of $\l2G $ by
 \begin{eqnarray*}
 \nn \l2G_p &:= &
 \{ \xi\in\l2G\vert \ \xi\ast p\in\l2G\}.
 \end{eqnarray*}
 Then $C_c(G)\subset \l2G_p $
 and for   every 
$f\in C_c(G)$,  a positive definite function $p=c^\pi_\va \in \L2G $ and $\et\in\l2G $ 
we have (using the Fubini theorem) that
\begin{eqnarray}\label{fastpet}
  \langle f\ast p,\et\rangle  &= &
 \int_G\int_G f(s)\langle \pi(s\inv t)\va,\va\rangle  ds\ol{\et(t)}dt\\
 \nn 
&= &
\int_G f(s)\int_G\langle \ol{\pi(t\inv s)\va,\va\rangle } \ol{\et(t)}dt ds
\\
\nn  
&= &
\int_G f(s) \ol{(\et\ast p)(s)}ds. 
%\\
%\nn&= &
%\langle f,\et\ast  p\rangle .
 \end{eqnarray}

%\end{remark}

\begin{definition}\label{denserho}\
Let $G $ be  a locally compact group. We call $G $ {right positive}, if 
for any positive definite function 
$p$ we have that 
\begin{eqnarray*}
 \nn \langle \et\ast p,\et\rangle  &\geq 0 &
 \text{ for every }\et\in \L2 G_p.
 \end{eqnarray*}
\end{definition}

\begin{lemma}\label{denseright}
For any  locally compact $\sigma$-compact  right positive group $G $ and any positive 
definite  function $p \in\L2G$ the subspace 
\begin{eqnarray*}
 \nn I_p &:= &
 \{ f\ast p+i f\vert f\in C_c(G)\}
 \end{eqnarray*}
is dense in $\L2G $.
 \end{lemma}

\begin{proof}

Let $\et\in I_p^\perp  $. 
Hence, using $p=\tilde p$ and applying the Fubini theorem, we obtain
\begin{eqnarray*}
{\ \nn0} &=&
\langle \et,f\ast p+if\rangle \\
\nn  
&= &
\int_G(\eta\ast p)(s)\overline{f(s)}ds-i\int_G\eta(s)\overline{f(s)}ds\\ %\langle \et\ast p,f\rangle+\langle \et,if\rangle \\
\nn  
&= &
\int_G((\eta\ast p)(s)-i\eta(s))\overline{f(s)}ds.%\langle \et \ast  p-i\et,f\rangle.
 \end{eqnarray*}
This implies that the measurable function $\et \ast  p-i\et $ is 
0 almost everywhere on each compact $K\subset G$. As $G$ is $\sigma$-compact, $\et \ast  p-i\et =0$ almost everywhere on $G$.  Hence $\et\ast  p=i\et $, $\et\ast  p\in \L2G $ and so 
\begin{eqnarray*}
 \nn \langle \et\ast p,\et\rangle  &= &
i \langle \et,\et\rangle.
 \end{eqnarray*}
Since $G$ is right positive it follows that $i\langle \et,\et\rangle \geq 0 $, 
giving  $\et=0 $.
\end{proof}
\begin{proposition}\label{abelianrightpos} 
Any separable, type I, unimodular group is right positive. 
%An abelian locally compact group and a second countable locally compact group 
%of Type I is  right positive.????
 \end{proposition}
\begin{proof} 
%Let $p\in\L2G $ be positive definite. 

%This implies that the    
%Fourier 
%transform $\widehat p \in\L2{\widehat{G}}$ is non-negative almost everywhere. 
%Therefore for any $\et\in\L2G_p $ we have that

%\begin{eqnarray*}\label{}
% \nn \langle \et\ast p,\et\rangle  &= &
%\langle \widehat{\et}\widehat{p},\widehat{\et}\rangle \\
%\nn\nn  
%& & 
%\int_{\widehat{G}}\vert \widehat{\et}(\widehat{x})\vert 
%^2\widehat{p}(\widehat{x})d\widehat{x}\\
%&\geq&
%0.
% \end{eqnarray*}

If $G$ is separable, type I and unimodular, we have a clear Plancherel picture of $G$ as follows. The unitary dual $\widehat{G}$ becomes a standard Borel space and there is a unique Borel measure $\mu$ on $\widehat{G}$ with the following property: for a fixed $\mu$-measurable cross-section $\xi \to \pi^\xi$ from $\widehat{G}$ to concrete irreducible unitary representations acting on $H_\xi$ we have
	$$\langle f_1, f_2\rangle = \int_{\widehat{G}}\text{Tr}(\widehat{f_1}^G(\xi) \widehat{f_2}^G(\xi)^*) d\mu(\xi),\;\; f_1, f_2 \in L^1(G)\cap L^2(G),$$
where
	\begin{equation}\label{eq-group-Fourier-transform}
	\widehat{f}^G(\xi) = \F^{G}(f)(\xi) := \int_G f(g)\pi^\xi(g)dg \in \B(H_\xi),\;\; f\in L^1(G)\cap L^2(G).
	\end{equation}
Thus, the group Fourier transform
	$$\F^G : L^1(G) \to L^\infty(\widehat{G}, d\mu; \B(H_\xi)),\; f\mapsto \F^G(f)$$ with $\F^G(f) = (\F^G(f)(\xi))_{\xi \in \widehat{G}} = (\widehat{f}^G(\xi))_{\xi \in \widehat{G}}$ extends to a unitary
	$$\F^G : L^2(G) \to L^2(\widehat{G}, d\mu; S^2(H_\xi)),\; f\mapsto \F^G(f).$$
Here, $S^2(H)$ is the space of Hilbert-Schmidt operators on a Hilbert space $H$.
If $\eta$, $p\in L^2(G)$ are such that $\eta\ast p\in L^2(G)$, then $\F^G(\eta\ast p)(\xi)=\F^G(\eta)(\xi)\F^G(p)(\xi)$ and 
$$\langle \eta\ast p,\eta\rangle =\int_{\hat G}\text{Tr}(\F^G(\eta)(\xi)\F^G(p)(\xi)\F^G(\eta)(\xi)^*)d\mu(\xi).$$
As $p\in L^2(G)$ is positive definite, by Lemma \ref{sqarer}, $p=c\ast \tilde c$ for square integrable function $c$ of positive type, giving
$\F^G(p)(\xi)=\F^G(c)(\xi)\F^G(c)(\xi)^*$, and hence $\F^G(p)(\xi)\geq 0$ almost everywhere. Therefore, $\langle\eta\ast p,\eta\rangle\geq 0$, giving the statement. 
\end{proof}

\begin{theorem}\label{l2l2}\marginpar{\tiny }
Let $G$ be a right positive, $\sigma$-compact, locally compact group and let
$u\in A(G)$ and $u_k\in A(G)$, $k\in \mathbb N$, be such that $u,\tilde u\in 
L^2(G)$
and $u_k,\tilde u_k\in L^2(G)$ for every $k$. Let $u=A(c)\ast c$ and 
$u_k=A_k(c_k)\ast c_k$, $k\in \mathbb N$,  be the canonical representations of  
$u$ and $u_k\in A(G)$ respectively, and $p=c\ast c$, $p_k=c_k\ast c_k$. 
%an element of
%$A(G)$ and $(u_k)_k$ be  a sequence in $A(G)$ such that $u,\tilde u$
%and $u_k,\tilde u_k\in L^2(G)$ for every $k$.  Let 
%$A_k, k\in\N,$ be a partial isometry in $VN(G)$ and $p_k$ be a positive
%definite continuous function in $L^2(G)$, with positive  square-root $c_k$
%in $L^2(G)$, such that  $u_k=A_k (p_k)=A_k (c_k)\ast
%c_k$, $p_k=A_k^*(u_k)$. Let  $u=A(p)=A(c)\ast
%c$, $p=A^*(u)$ for a partial isometry $A\in VN(G)$ and a
%continuous positive definite function $p\in L^2(G)$ with square
%root $c$. 
Suppose that
\begin{enumerate}
\item $\no{u_k}_{A(G)}\to \no u_{A(G)}$;
\item $\|u_k-u\|_2\to 0$;
\item   $p_k\to p$  weakly in $L^2(G)$.
\end{enumerate}
 Then
 $\|c_k-c\|_2\to 0$ and  $\|A_k(c_k)-A(c)\|_2\to 0$.  In particular, 
$\|p_k-p\|_{A(G)}\to 0$ and $\|u_k-u\|_{A(G)}\to 0$.

 \end{theorem}
\begin{proof}

We shall use an idea from the  proof of 
\cite[Proposition 5.1]{can_haag} to conclude that $ (c_k)_k$ converges to $c$ in 
$L^2(G)$.

We have that
 \begin{equation*}
  \no {u_k}_2=\no {A_k(p_k)}_2\leq \no {p_k}_2=
  \no {A_k^*(u_k)}_2\leq \no {u_k}_2
  ,\text{ for all }k.
  \end{equation*}

Hence $\no{u_k}_2= \no{p_k}_2$
 for all $k$ and similarly $\no
u_2=\no p_2$. The same is true for the $A(G)$-norms. Since $\|u_k-u\|_2\to 0$, 
we  have the convergence $\no {p_k}_2\to\no 
p_2$.
Together with the weak-convergence of $(p_k)_k$ to $p$, this shows
that $\|p_k-p\|_2\to 0$. 

We are going to show next that
$\rh(c_k)f$ converges to $\rh(c)f$ in $L^2(G)$ for any $f\in\CC
G$.
Since each $\rh(p_k)$ and $\rh(p)$ are self-adjoint
and positive, the operators $\rh(p_k)+i\I$ and   $\rh(p)+i\I$ have bounded 
inverses; here we write $\I$ for the identity operator on $L^2(G)$. Furthermore,
$$\no{(\rh(p_k)+i\I)\inv}_{\text{op}}\leq 1, \forall k \text{ and
} \no{(\rh(p)+i\I)\inv}_{\text{op}}\leq 1.$$ Let $f\in\CC G$ and
$g=
 (\rh(p)+i\I)f\in L^2(G)$. As $p_k\to p$ in $L^2(G)$, we have that
 $$\|\rho(p_k)f-\rho(p)f\|_2=\|f\ast p_k-f\ast p\|_2\leq\|f\|_1\|p_k-p\|_2\to 
0$$
 and
 \begin{equation*}
 [(\rh(p_k)+i\I)\inv-(\rh(p)+i\I)\inv]g=(\rh(p_k)+i\I)\inv(\rh(p)-\rh(p_k))f\to 
0.
 \end{equation*}
  Since the operators
 $(\rh(p_k)+i\I)\inv$
 are uniformly bounded  and since, by Lemma \ref{denseright}, the subspace
 $\{(\rh(p)+i\I)f: { f\in C_c(G)}\}$ is dense in
 $L^2(G)$, it follows that
 \begin{equation}\label{polyn}
 (\rh(p_k)+i\I)\inv{ g}\to (\rh(p)+i\I)\inv { g} \quad 
\text{ for all } { g}\in
 L^2(G).
 \end{equation}
 A similar proof works for $(\rh(p_k)-i\I)\inv$.

 Define now two continuous functions $h$, $q: [0,\iy)\to\R$ by letting 
$$h(t):=\left\{%
\begin{array}{ll}
   \sqrt t-t, & \hbox{ if } t\leq 1\\
    0, & \hbox{ if } t\geq 1\\
\end{array}%
\right.\text {and }
q(t):=\left\{%
\begin{array}{ll}
   1, & \hbox{ if } t\leq 1\\
    \frac 1 {\sqrt t}, & \hbox{ if } t\geq 1\\
\end{array}%
\right.$$

Then  $h$ and $q \in C_0([0,\iy))$, and
 \begin{equation*}
\sqrt t=h(t) + q(t) t,\quad t\geq 0.
 \end{equation*}
%It follows from \cite{RS}, Theorem VIII.20, that

 By the Stone-Weierstrass theorem, the
 polynomials in
$(x+i)\inv$ and
 $(x-i)\inv$ are dense in  $C_{0}([0,\iy))$. Thus, given $\ve>0$,
we can find a polynomial $P(s,t)$  such  that
\begin{equation*}
\left\vert q(x)-P\left(\frac 1 {x+i},\frac 1{x-i}\right)\right\vert<\frac \ve 
3\quad
\text{ for all }x\geq 0.
\end{equation*}

Therefore,
\begin{equation*}
\no{
q(\rh(p_k))-P((\rh(p_k)+i\I)\inv,(\rh(p_k)-i\I)\inv)}_{\text{op}}<\frac
\ve 3, \quad \text{ for all } k,\end{equation*}
 and
 \begin{equation*}
\no{
q(\rh(p))-P((\rh(p)+i\I)\inv,(\rh(p)-i\I)\inv)}_{\text{op}}<\frac
\ve 3.
\end{equation*}
It follows from (\ref{polyn}), that
\begin{equation*}
P((\rh(p_k)+i\I)\inv,(\rh(p_k)-i\I)\inv) f\to
P((\rh(p)+i\I)\inv,(\rh(p)-i\I)\inv) f,\text{ for all } f\in
L^2(G).
 \end{equation*}

Thus for any $f\in L^2(G)$, there  exists an $N_1(f)\in\mathbb N$ such that
\begin{equation}\label{krh}
\no{q(\rh(p_k))f-q(\rh(p))f}_2\leq \ve
\end{equation}
for any  $k\geq
N_1(f)$.

Since $\rh(p_k)f\to\rh(p)f$ for all $f\in \CC G$,
there exists an $N_2(f)\in\mathbb N$ such that
\begin{equation*}
\no{\rh(p_k)f-\rh(p)f}_2\leq \ve, \text{ for all }k\geq N_2(f).
 \end{equation*}

Finally,  for $f\in \CC G$ and  $k\geq \text{max}\{N_1(\rh(p)f), N_2(f)\}$
(since $\no {q(\rho(p_k))}_{\text{op}}\leq \no q_{\iy}=1$ for all
$k$) we have that
\begin{eqnarray*}
  \no{q(\rh(p_k))(\rh(p_k)f)-q(\rh(p))(\rh(p)f)}_2 &\leq&
  \no{q(\rh(p_k))(\rh(p_k)f)-q(\rh(p_k))(\rh(p)f)}_2\\
  &&+\no{q(\rh(p_k))(\rh(p)f)-q(\rh(p))(\rh(p)f)}_2 \\
   &\leq & \no{\rh(p_k)f-\rh(p)f}_2 \\
  &&+\no{q(\rh(p_k))(\rh(p)f)-q(\rh(p))(\rh(p)f)}_2
   \\
   &\leq&\ve+\ve.
\end{eqnarray*}
Similar arguments applied to $h$ instead of $q$ give us that
\begin{equation}\label{conno}
h(\rh(p_k))f\to h(\rh(p))f\quad \text{ for all } f\in C_c(G) .
\end{equation}

Together this shows that
\begin{equation*}
\sqrt{\rh(p_k)}f\to\sqrt{\rh(p)}f, \text{ for all }f\in \CC G.
 \end{equation*}

But by (\ref{squareroot}), $ p_k=c_k\ast c_k$, $p=c\ast c$ and
$\sqrt{\rh(p_k)}f=\rh(c_k)f$, for all $k$, and $\sqrt{\rh(p)}f=\rh(c)f$, 
$f\in\CC G.$
Hence
\begin{equation*} f\ast c_k\to f\ast c, \text{ for all
}f\in \CC G.
\end{equation*}

Therefore, for every $f,g\in \CC G$, it follows that
\begin{equation}\label{convcon}
 \langle c_k,f^*\ast g\rangle=\langle f\ast c_k, g\rangle\to \langle f\ast c, 
g\rangle=
 \langle c,f^*\ast g\rangle.
 \end{equation}
Here $f^*(s)=\bar f(s^{-1})\Delta(s^{-1})$ and $\Delta$ is the modular 
function. 
Since the functions $f^*\ast g,  f,g\in \CC G,$ generate a dense subspace in
$L^2(G)$ and since by assumption
 \begin{equation*}\no
{c_k}^2_2=p_k(e)=\no {p_k}_{A(G)}= \no {u_k}_{A(G)}\to 
\no{u}_{A(G)}=\no{p}_{A(G)}=p(e)=\no
{c}^2_2,
\end{equation*}
it follows from (\ref{convcon}) that $(c_k)_k$ converges 
to $c$ weakly in $L^2(G)$ and finally also in norm, as
\begin{eqnarray*}
\|c_k-c\|_2^2=\|c_k\|_2^2+\|c\|_2^2-2\Re{\langle c,c_k\rangle}|\to 0.
\end{eqnarray*}

The sequence $(A_k)_k $, being uniformly bounded by 1,  admits a weakly 
convergent  subnet 
$(A_{k_i})_i $. Let $A_\infty\in VN(G)$ be its limit.  We are going to show 
that 
$A_\iy(p)=A(p)$ and
$A_\iy^*(u)=A^*(u)$.

Indeed, for $f\in\CC G$, we have that
\begin{eqnarray*}
  \langle u_{k_i}, f\rangle &=&\langle A_{k_i}(p_{k_i}),f\rangle\\
   &=& \langle p_{k_i},A_{k_i}^*(f)\rangle \\
 \downarrow & & \downarrow \quad (\text{since } p_{k_i}\to p, u_{k_i}\to u \in 
L^2(G))\\
   \langle u, f\rangle&=& \langle p,A_\iy^*(f)\rangle=\langle A_\iy(p),f\rangle.
\end{eqnarray*} 

Hence $u=A(p)=A_\iy (p)$. Similarly,
\begin{eqnarray*}
\langle A^*(u),f\rangle&=&\langle p,f\rangle=\lim_k\langle 
p_k,f\rangle=\lim_k\langle  A_k^*(u_k),f\rangle=\lim_{k_i}\langle  
u_{k_i},A_{k_i}(f)\rangle\\&=&\langle u,A_\iy(f)\rangle=\langle 
A_\iy^*(u),f\rangle
\end{eqnarray*}
and hence 
$A^*(u)=A^*_\iy (u).$

Furthermore,  for any $x\in G$,
  \begin{eqnarray*}
    \sp{A_\iy(c)-A(c)}{\la(x)c} &=& ((A_\iy(c)-A(c))\ast c)(x) \\
    &=&(A_\iy(c)\ast c)(x)-(A(c)\ast c)(x)\\
    &=&A_\iy(c\ast c)(x)-A(c\ast c)(x)\\
    &=&A_\iy(p)(x)-A(p)(x)\\
     &=& u(x)-u(x) \\
    &=& 0. \\
  \end{eqnarray*}
This relation tells us that $A_\iy(c)-A(c)$ is  contained in the
  orthogonal complement to $\la(C_c(G))(c)$.
On the other hand, since $A_\iy, A \in VN(G)$,  $A_\iy$, $A$ are strong limits of nets contained in  $\la(C_c(G))$. Consequently,  
$A(c)  $, $A_\iy(c)\in \overline{\la(\CC G)(c)}$ and  
therefore $A_\iy(c)-A(c)=0 $.

Observe next that $A^*A(p)=p$. As $p=c\ast c=c\ast\tilde c$, it follows from (\ref{leri}) and
(\ref{l2ag}) that
\begin{equation}\label{identi}
p(x)=A^*A(p)(x)=(A^*A(c)\ast\tilde c)(x)
\end{equation}
for almost all $x\in G$. Now, since both $p$ and $A^*A(c)\ast\tilde c$ are continuous, we have the equality everywhere on $G$. 

Similarly, $p_k(x)=(A_k^*A_k(c_k)\ast c_k)(x)$ for all $x\in G$, $k\in\mathbb N$.

%Let us recall the following formula. For $u=a\ast \tilde b\in
%A(G)\cap L^2(G)$ and $E\in VN_\la(G)$ we have by (\ref{leri}) and
%(\ref{l2ag}),
 %that
 %\begin{equation*}\label{}
%E(u)=E(a)\ast \tilde b\end{equation*} as elements in $L^2(G)$.
%Hence if $E(u)=u$, we have for almost all $x\in G$, that
%\begin{equation*}\label{}
%u(x)=E(u)(x)=E(a)\ast \tilde b(x)\end{equation*} and since both
% functions $u$ and $E(a)\ast\tilde b $ are continuous, it follows
% that
%\begin{equation}\label{identi}
%u(x)=E(a)\ast \tilde b(x),\quad \text{ for all } x\in G.
%\end{equation}
Hence,
%, since $A_k^*A_k(p_k)=p_k$ for all $k$ and   $A^*A(p)=A_\iy^*A_\iy 
%(p)=p$, 
%it follows from (\ref{identi}) that
\begin{eqnarray*}
 \no{A_k(c_k)}_2^2 &=&\langle A_k(c_k),A_k(c_k)\rangle \\
  &=& \langle A_k^*A_k(c_k),c_k\rangle \\
 &=& (A_k^*A_k(c_k)\ast c_k)(e) \\
   &=& p_k(e)=\no{p_k}_{A(G)}\\
   &\to & \no{p}_{A(G)}=p(e)\text{ (as }k\to\iy)\\
   &=& (A^*A(c)\ast c)(e)=\no{A(c)}^2_2.
\end{eqnarray*}

Therefore, the weakly convergent net $(A_{k_i}(c_{k_i}))_i$ converges in
fact in norm to $A_\iy(c)=A(c) $, from which we can conclude that the convergence holds for the entire sequence $(A_k(c_k))_k$. In fact, otherwise, there exist $\varepsilon>0$ and a subsequence $(A_{k(n)}(c_{k(n)}))_n$ such that $\|A_{k(n)}(c_{k(n)})-A(c)\|>\varepsilon$. Repeating the previous arguments, we find a subsequence of $(A_{k(n)}(c_{k(n)}))_n$ that converges to $A(c)$ in norm, contradicting the choice of  $(A_{k(n)}(c_{k(n)}))_n$.

\end{proof}

\marginpar{\tiny }
\begin{cor}\label{ckak}
 Suppose that  $u, u_k\in A(G)$, $k\in \mathbb N$,   are such that 
$\|u_k-u\|_{A(G)}\to 0$ and $\supp {u_k}\subset K$ for all $k\in \mathbb N$ and 
some compact set $K\subset G$. 

Let $u_k=A_k
(c_k)\ast c_k$  and $u=A(c)\ast c$ be  the canonical representations of $u_k$ and $u$ respectively. 
   Then $\|c_k-c\|_2\to 0$ and $\|A_k(c_k)-A(c)\|_2\to 0$.
 
\end{cor} 

\begin{proof}
As $\|\tilde w\|_{A(G)}=\|w\|_{A(G)}$ for all $w\in A(G)$ (\cite[Remark 2.10]{E}),  we have $\|\tilde u_k-\tilde u\|_{A(G)}\to 0$ and if $p_k$ and $p$ are the absolute values of $\tilde u_k$ and $\tilde u$ respectively, applying \cite[III, Proposition 4.10]{ta}, we obtain $\|p_k-p\|_{A(G)}\to 0$. In
 particular,   $ p_k$ tends to $p$ and $u_k$ to $u$ uniformly on $G$. 
Now, by the assumption, we can find a common compact subset that contains the supports of  $u$, $\tilde u$, $u_k$ and $\tilde u_k$, $k\in\mathbb N$. Therefore, $\|u_k- u\|_2\to 0$. 
As $p_k=A_k^*(u_k)$, we have that the 
sequence $(p_k)_k$ is bounded in $L^2(G)$. The uniform convergence  $p_k\to p$  on $G$, gives $\langle p_k-p,f\rangle\to 0$ for any $f\in C_c(G)$. As $C_c(G)$ is dense in $L_2(G)$ and $(p_k)_k$ is bounded, this implies that $p_k\to p$ weakly in $L^2(G)$.  We have thus verified all the
conditions of Theorem~\ref{l2l2}, which gives us the statement. 

%Hence there exists a sub-sequence of 
%$(p_k)_k$ that converges weakly to some $p_\infty\in L^2(G)$. Since $p_k$ tends
% uniformly to $p$, it follows that $p_\iy=p$. As a consequence, we obtain that 
% $p_k$ converges weakly to $p$. We have thus verified all the
%conditions of Theorem~\ref{l2l2}, which gives us the statement.

 %Let $K$ be a compact subset of $G$ such that $K=K\inv$ contains the supports of 
%all
% our functions
 %$u_k$ and $u$. Since now $u_k$ for each $k$ and $u$ itself are in $\CC K$,
 %so do the
% functions $\tilde u_k$ and $\tilde u$. Furthermore, since
% $u_k$ tends to $u$ in
% $A(G)$ and since  by \cite[Remark 2.10]{E}, for every $w\in A(G) $,  $\|\tilde 
%w\|_{A(G)}=\|w\|_{A(G)}$,  we have  that $\tilde u_k$ tends to $\tilde u$ and 
%the
 %absolute value $p_k$ of $\tilde u_k$ tends in $A(G)$ to the absolute
 %value $p$ of $\tilde u$ by \cite[III, Proposition 4.10]{ta}. In
% particular $u_k$ tends to $u$ and $p_k$ to $p$ uniformly and therefore
 %$u_k$ converges to $u$ in $L^2(G)$. As $p_k=A_k^*(u_k)$, we have that the 
%sequence $(p_k)_k$ is bounded in $L^2(G)$. Hence there exists a sub-sequence of 
%$(p_k)_k$ that converges weakly to some $p_\infty\in L^2(G)$. Since $p_k$ tends
% uniformly to $p$, it follows that $p_\iy=p$. As a consequence
% $p_k$ tends weakly to $p$. We have thus verified all the
%conditions of Theorem~\ref{l2l2}.
\end{proof}

\subsection{A characterization of local spectral sets} \label{characterization}

%\begin{definition}
%\begin{definition}\label{ladef}
\rm  Let $t\in G, \xi\in\L2G $ and $T\subset G$.  For simplicity of notation, write
\begin{eqnarray*}
 \nn t\cdot \xi &:= &
 \la(t)\xi.
 \end{eqnarray*}
and set 
 \begin{eqnarray*}
 \nn \LA( 
T\cdot \xi) &:= &
 \overline{\text{span}(T\cdot \xi)}.
 \end{eqnarray*}

For a closed subspace $ E $ of $ \L2G $,
let $ P_E $ be the orthogonal projection onto $ E $ and write 
%\begin{eqnarray}\label{projection}
  $P_{T\cdot \al}$ for the projection 
  %&:= &
$P_{\LA(T\cdot\al)}$.
We are now ready to prove the  main result. 
\begin{theorem}\label{spectralstrcon}

Let  $G$ be a locally compact,  $\sigma$-compact, right positive group and let $S$ be a closed 
subset
of $G$. Then $S$ is a local spectral set if and only if for any $u\in k(S)\cap 
C_c(G)$ 
there exist 
%a Hilbert space $\mathcal H$,
a representation $u(s)=\langle 
d, \lambda(s) c\rangle$, $c, d\in 
L^2(G)$, a sequence $(c_k)_k$  in $L^2(G)$ and a sequence 
$(S_k)_k$ of closed neighborhoods  of $S$ such that 
\begin{equation}\label{convergence}\lim_k c_k=c \text{ and } 
 \lim_k P_{S_k\cdot c_k}(d) =0.
 \end{equation}

Moreover, if $S$ is a set of local spectral synthesis, there is a sequence $(c_k)_k\subset L^2(G)$ of functions of positive type satisfying (\ref{convergence}), with $c, d=A(c)\in L^2(G)$ from the canonical representation $u(s)=\langle d,\lambda(s)c\rangle=(A(c)\ast\tilde c)(s)$. 
%is the canonical representation of $u$. we can take the canonical representation of $u(s)=(A(c)\ast c)(s)=\langle A(c), \lambda(s)c\rangle$   and choose  a sequence $(c_k)_k\subset L^2(G)$ of function of positive type  satisfying the above condition with $d=A(c)$.
\end{theorem}

\begin{proof}
Assume $u\in k(S)\cap C_c(G)$ has a representation $u(s)=\langle d, \lambda(s)c\rangle$, $c$, $d\in L^2(G)$. 
Let $(c_k)_k\subset  L^2(G)$ and let $(S_k)_k$ be a sequence of 
neighborhoods of $S$, which satisfy the conditions of the theorem. Set 
$d_k=d-P_{S_k\cdot c_k}(d)$. Then $d_k\in \Lambda(S_k\cdot c_k)^\perp$ and 
hence $u_k(s):=\langle d_k,\lambda(s)c_k\rangle$ vanishes 
on $S_k$. Moreover,  
\begin{eqnarray*}
\|u-u_k\|_{A(G)}&\leq&\|\langle d_k, \lambda(s)(c-c_k)\rangle\|_{A(G)}+\|\langle d-d_k, \lambda(s)c_k\rangle\|_{A(G)}\\&\leq&\|c-c_k\|_2\|d_k\|_2+\|c_k\|_2\|d-d_k\|_2\to 0
\end{eqnarray*}
showing that $S$ is a set of local spectral synthesis.

Suppose now that $S$ is a set of local spectral synthesis and take $u\in k(S)\cap C_c(G)$. Let $K $ be  a 
compact neighborhood of the support of $u$. Since
$S$ is local spectral, there exists a sequence $(u_k)_{k}$ in $A(G)$,   such 
that,
for every $k$, $\supp{u_k}\subset K$,
$u_k$ vanishes on a neighborhood $S_k$ of $S$  and $u_k$ converges to $u$ in $A(G)$. 

Consider the canonical representation $u_k=A_k(c_k)\ast c_k$ and set $p_k=c_k\ast c_k$ and $d_k=A_k(c_k)$. 
%Using Proposition \ref{decomp}, we
%write for every $ k$, $ u_k=A_k(c_k)\ast c_k $ where $ c_k\in L^2(G)$ is  the
%positive square root of $p_k$, the absolute value of $\tilde u_k$,
%and   $A_k\in VN(G)$ is  a partial isometry, such
%that $ p_k=A_k^* A_k(p_k) $. Let  $ d_k=A_k(c_k) $. 
Then  $ d_k\in
\LA(S_k\cdot c_k)^\perp$, since $ u_k $ vanishes on $S_k $.
Therefore
$ P_{S_k\cdot c_k}(d_k)=0 $ for
every $ k$. By Corollary \ref{ckak},  $ \lim_{k }c_k=c, \lim_{k}d_k=d  $ and 
hence $
\lim_{k}P_{S _k\cdot c_k}(d)=0 $.
\end{proof}

\subsection{The abelian case}\label{abelian}

% \subsection{}
 The goal of this section is to provide a refinement of the characterization of local spectral sets from Section \ref{characterization} in the case of abelian groups.  For an abelian locally compact group $G$ let $\hat G$ be the dual of $G$. We write $\hat a$ for the Fourier transform of $a\in L^2(G)$.
 %that is $\hat a(\chi)=\int_G a(s)\overline{\chi(s)}ds$, $\chi\in\hat G$.  

\begin{lemma}\label{Ltwo norm}
Let $G $ be an abelian locally compact group.
Let $c,d\in \L2G $, such that $\hat d=\ps\hat c $ for some   $\ps\in \L\iy 
{\hat G}$. 
Then for any  subset $S $ of $G $ we have that 
\begin{eqnarray*}
 \no{P_{S\cdot c}(d)}_2\leq \no\ps_\iy\no{P_{S\inv\cdot d}(c)}_2.
 \end{eqnarray*}
\end{lemma}
\begin{proof} 
We shall use the Plancherel theorem and consider therefore  the action of $G$ on the Hilbert 
space $\L2{\hat G} $: 
%The group $G $ acts   by multiplication on the elements of 
%$\L2{\hat G} $:
\begin{eqnarray*}
 t\cdot \xi(x)=\ch_t(x)\xi(x), x\in\hat G, \xi\in\L2{\hat G}, t\in G,
 \end{eqnarray*}
 where $\ch_t $ denotes the character of $\hat G $ defined by $t\in G $.
 
 Set $\xi:=\hat c $, $\et:=\hat d $. 
 We have 
% \begin{eqnarray}%\label{rel ps xi}
$$ \et=\ps \xi.$$
 %\end{eqnarray}

By the 
definition 
of $\LA(S\cdot \xi)  $ the elements  of the space $P_{\LA(S\cdot \xi)}(\L2{\hat 
G}) $ are of   the form $\va \xi $ for some measurable function 
$\va:\hat G\to \C $ and $\va\xi $ is an $L^{2} $-limit of functions 
$\va_k \xi $, where  
\begin{eqnarray*}
 \va_k=\sum_{j=1}^{m_k} c_j^k \ch_{s^k_j}
 \end{eqnarray*}
for some constants $c_j^k\in\C $ and $s_j^k\in S $.
Hence, for such a $\va \xi \in \LA(S\cdot \xi)$ we have that 
\begin{eqnarray}\label{va et in La}
\ol\va \et &=&
\ol\va \ps \xi= 
\ps \ol\va \xi\\
\nn  &= &
\ps (\limk \ol{\va_k} \xi)= 
\limk \ol{\va_k} \ps \xi\\
\nn  &= &
\limk \ol{\va_k} \ \et\in 
\LA(S\inv\cdot \et).
 \end{eqnarray}

Since 
\begin{eqnarray} \label{norm of proj}
 \no{P_{\LA(S\cdot \xi)}(\et)}_2=\sup\{\vert \langle \et,\va 
\xi\rangle\vert ; \va\xi \in \LA(S\cdot \xi),\no{\va \xi}_2=1\},
 \end{eqnarray}
it follows that for any $\va \xi $ of norm 1:
\begin{eqnarray*}
 \vert \langle \et,\va \xi\rangle\vert&=&\vert \langle \ol\va 
\et,\xi\rangle\vert\\
  \nn  &= & 
\vert \langle P_{\LA(S\inv\cdot \et)}(\ol\va \et),\xi\rangle\vert\\
\nn  &= &
\vert \langle \ol\va\et,P_{\LA(S\inv\cdot \et)}(\xi)\rangle\vert\\
 \nn  &\leq & 
 \no{\ol\va \ps \xi}_2\no{P_{\LA(S\inv\cdot \et)}(\xi)}_2\\
 \nn  &\leq &
 \no \ps_\iy\no{\ol\va \xi}_2\no{P_{\LA(S\inv\cdot\et)}(\xi)}_2\\
 \nn  &= &
\no \ps_\iy\no{P_{\LA(S\inv\cdot \et)}(\xi)}_2.
 \end{eqnarray*}
 Hence
 \begin{eqnarray*}
 \no{P_{\LA(S\cdot \xi)}(\et)}_2\leq \no 
{{\ps}}_\iy\no{P_{\LA(S\inv\cdot \et)}(\xi)}_2.
 \end{eqnarray*}

\end{proof}
\begin{corollary}\label{proj norm} 
Let $G $ be an abelian locally compact group. Let $c, d\in \L2G $ such that 
$\hat d=\ps \hat c$, where $\ps:\hat G\to\C $ is  a measurable function of 
absolute value equal to 1 on the support of $\hat c $. 
Then for any  subset $S $ of $G $ we have that 
\begin{eqnarray*}
 \no{P_{\LA(S\cdot c)}(d)}_2=\no{P_{\LA(S\inv\cdot d)}(c)}_2.
 \end{eqnarray*}
\end{corollary}

\begin{proof} As $\hat d=\ps \hat c$ and  $\hat c=\ol\ps\hat d $, Lemma \ref{Ltwo 
norm} gives us 
\begin{eqnarray*}
 \nn &\no{P_{\LA(S\cdot c)}(d)}_2 \leq  \no{P_{\LA(S\inv\cdot d)}(c)}_2&
 \\
 \nn&\leq
 \no{P_{\LA(S\cdot c)}(d)}_2\\
 &\Rightarrow\\
 &\no{P_{\LA(S\cdot c)}(d)}_2 =\no{P_{\LA(S\inv\cdot d)}(c)}_2.
 \end{eqnarray*}

\end{proof}

\begin{lemma}\label{Skukiszero}
Let $G $ be  a locally compact, $\sigma$-compact, abelian group. Let $(S_k)_k $ be a  sequence 
of closed subsets of $G $. 
Let $ (u_k)_{k}$ 
be a 
converging sequence in $A(G)$ with limit $u\in A(G)\cap L^2(G)$, 
such that 
$u_k(S_k)=\{ 0\} $ for every $k $.  Let $u=d\ast c $ be the canonical 
form of $u\in\l2G $. Then $\limk P_{S_k\cdot c}(d)=0 $.

 \end{lemma}
 \begin{proof} 

Choose a sequence $(v_i)_i\in C_c(G)\cap A(G) $, such that 
$\no{v_i}_{A(G)}\leq 1, i\in\N, $ and such that $\lim_{i\to\iy}v_iu=u $ in 
$A(G) $. Then the sequence $(v_iu_k)_k\in\N, $ converges in $A(G) $ and  in $\L2G $ to 
$v_iu  $. We write the elements $u^i_k:=v_i u_k $ in the canonical form 
$u_k^i=d_k^i\ast  c^i_k$, with $p^i_k=\widehat{c_k^i}\geq 0 $ and 
$\widehat{d^i_k}=\va^i_k \widehat{c^i_k} $ and ${\vert \va^i_k\vert 
}=1_{p^i_k}, k\in\N$. Similarly for $v_iu=d^i\ast c^i $. Then  Theorem 
\ref{l2l2} tells us   that the sequence 
$(c_k^i)_k $ converges in $\L2G $ to $c^i $ and the sequence $(d^i_k)_k $ to $ 
d^i $. Since $u^i_k(S_k)=\{ 0\} $ for any $k,i\in\N $, we have that
\begin{eqnarray*}
 \nn P_{S_k\cdot c_k^i}(d^i_k) &= &
 0, k,i\in\N.
 \end{eqnarray*}
 Therefore, since $\limk d^i_k=d^i $ for any $i $, it follows that 
 \begin{eqnarray*}
 \nn \limk P_{S_k\cdot c_k^i}(d^i) &= &
 0.
 \end{eqnarray*}
Hence  
by Corollary \ref{proj norm} 
 \begin{eqnarray*}
 \nn \limk \Vert{P_{S_k^{-1}\cdot d^i}(c_k^i) }\Vert_{2 }  &= 
&\limk\no{P_{S_k\cdot 
c_k^i}(d^i)}_{2}\\
 \nn  
&= & 
 0,i\in\N. 
 \end{eqnarray*}
Finally
 \begin{eqnarray*}
 \nn   \limk\no{
 P_{S_k\cdot c^i}(d^i)}_2&=&
 \limk\no{
 P_{S_k^{-1}\cdot d^i}(c^i)}_2
  \\
 &= &
 \limk\no{P_{S_k^{-1}\cdot d^i}(c_k^i)}_2=
 0.
 \end{eqnarray*} 
   
 Now, since $u\in\L2G $,  again by Theorem \ref{l2l2},
$\limi c^i=c, \limi 
d^i=d.  $  Therefore  it follows as before that 
\begin{eqnarray*}
 \nn  \Vert{P_{S_k\cdot c}(d) }\Vert_2 &\leq &
 \Vert{P_{S_k\cdot c}(d^i-d) }\Vert_2 +\Vert{P_{S_k\cdot c}(d^i) }\Vert_2\\
 \nn  
&= &
\Vert{P_{S_k\cdot c}(d^i-d) }\Vert_2 +\Vert{P_{S_k\inv\cdot d^i}(c) }\Vert_2\\
\nn  
&\leq &
\Vert{P_{S_k\cdot c}(d^i-d) }\Vert_2 +\Vert{P_{S_k\inv\cdot d^i}(c-c^i)\Vert_2 
}+\Vert{P_{S_k\inv\cdot d^i}(c^i)\Vert_2 
}\\
\nn  
&= &
\Vert{P_{S_k\cdot c}(d^i-d) }\Vert_2 +\Vert{P_{S_k\inv\cdot d^i}(c-c^i)\Vert_2 
}+\Vert P_{S_k\cdot c^i}(d^i)\Vert_2. 
\end{eqnarray*}
Therefore
\begin{eqnarray*}
 \nn \Vert{P_{S_k\cdot c}(d) }\Vert_2 &\leq &
  \Vert{c-c^i }\Vert_2+ \Vert{d-d^i }\Vert_2+\Vert P_{S_k\cdot c^i}(d^i)\Vert_2
 \end{eqnarray*}
and so for every $i\in\N $
\begin{eqnarray*}
 \nn  {\limk \Vert P_{S_k\cdot c}(d) }\Vert_2&\leq&
\limk\no{
 P_{S_k\cdot c^i}(d^i)}_2+\Vert{c-c^i }\Vert_2+ \Vert{d-d^i }\Vert_2\\
 \nn  
&= &
\Vert{c-c^i }\Vert_2+ \Vert{d-d^i }\Vert_2.
\end{eqnarray*}
This shows that 
$\limk P_{S_k\cdot c}(d)=0$ 

\end{proof}

 \begin{cor}\label{characterization spectral G abelian}
Let  $G$ be a locally compact, $\sigma$-compact, abelian group and let $S$ be a closed subset
 of $G$. Then $S$ is a  spectral set if and only if for every $u\in 
k(S)\cap C_c(G) $ ($u=d\ast c $ being its canonical expression)  we 
have   a decreasing 
sequence $(S_k)_k $ of closed neighborhoods of $S $ 
%-such that   for the decreasing sequence of closed subspaces $\S_k:=\LA(S_k\cdot c) $ we have 
 such that $\limk P_{S_k\cdot c}(d)=0 $.

 \end{cor}
 \begin{proof} Suppose that $S $ is spectral. 
Let $u\in k(S)\cap C_c(G)$ and $u=d\ast c $ be its canonical 
expression. Since $S $ is spectral, 
there 
exists a sequence of neighborhoods $(S_k)_{k} $ of $ S$, which can be chosen to be decreasing by taking intersections,  
 and a sequence 
$(u_k)_k\subset A(G) $, 
such that 
$\limk u_k=u $ in $ A(G)$ and such that $u_k $ vanishes on $S_k, 
k\in\N. 
$ By Lemma \ref{Skukiszero}, $\limk P_{S_k\cdot c}(d)=0 $.

Conversely, if $u=d\ast c \in A(G)\cap C_c(G)$ vanishes on $S $ and if $\limk P_{S_k\cdot c}(d)=0 $,
%$\limk \LA(S_k\cdot c)(d)=0$, 
then  the sequence $(u_k:=(d-P_{S_k\cdot c}(d))\ast c)_k $ converges in 
$A(G) 
$ to $u $ and $u_k(S_k)=\langle (d-P_{S_k\cdot c}(d),S_k\cdot c\rangle =\{ 0\} $ for 
any 
$k\in\N $. Therefore, $S$ is a local spectral set. As $G$ is abelian, it is a spectral set. %\marginpar{\tiny need an argument for decreasing sequence }

%If we take any $u\in A(G) $ vanishing on $S $, then we use a bounded 
%approximate unit in $A(G) $ to  approximate $u $ with 
%a sequence $(u^i=c^i\ast d_i)_i\subset A(G)\cap C_c(G)$  vanishing on $S 
%$. Then we can approximate 
%each $u^i $   with a sequence $ (u_k')_k\subset j(S)$ and finally we see that 
%$u  $ itself is  a limit of functions contained in $j(S) $.
\end{proof}
 
  \section{strongly spectral  sets and the union problem}\label{strong}

In this section we will introduce a new class of sets which includes Ditkin 
sets 
and show that it is closed under the operation of forming finite unions.

{It is easy to see that the union of two disjoint (local) spectral  sets  and 
of 
two Ditkin sets is  a (local) spectral set.
The question about  the union of any two non-disjoint (local) spectral sets  
was raised  in the paper \cite{H1} of C. Herz (for abelian groups) and three  
years later again by H. Reiter in  \cite{Re}. The problem  remains open. }
   
\medskip

% Let $S $ be a closed subset of the locally compact group $G $.
% We define the ideal $k(S) $ of $A(G) $ by
% \begin{eqnarray*}\label{}
% \nn k(S) &= &
% \{ u\in A(G)\vert  u(s)=0, s\in S\}.
% \end{eqnarray*}

{\begin{definition}\label{strsp}$ $
We say that a closed subset $ S $ of $ G $ is \textit{strongly 
(local) 
spectral}, if for every compact subset $T $ of $G $ and any $ f\in k(S\cup T)$ 
($f\in k(S\cup T)\cap C_c(G)$) and any $ 
\ve>0 $ 
there exists an element  $ g_\ve $ in $j(S) \cap k(T)$  such that  $ 
\noag{f-g_\ve}<\ve$ 
.

 \end{definition}}   
   
   \begin{remark}\rm 
Any Ditkin set  is obviously strongly  spectral, but we do not know whether the converse is true. 

     We note that if the group $G$ is such that $u\in\overline{uA(G)}$ for each $u\in A(G)$ (for instance, $G$ is amenable or, more generally,  when $A(G)$ has an (unbounded) approximate unit),  then $S\subset G$ is strongly spectral set if and only if, for any $f\in k(S)$ and any $\ve>0$, there exists a function $g_\ve\in j(S)$ vanishing on $\nul(f)$ and satisfying $\noag{f-g_\ve}<\ve$. In fact,  in the latter case, if $T\subset G$ is compact and $f\in k(S\cup T)$ then, as $T\subset\nul(f)$, the function $g_\ve$ vanishes on $T$. To see the converse, first note  that the set of compactly supported functions in $A(G)$ is dense in $A(G)$. Assume that $S$ is strongly spectral and let $f\in k(S)$. Given $\ve>0$, there exist a compactly supported $h\in A(G)$ and $\tilde g_\ve\in j(S)\cap k(\nul(f)\cap\supp{h})$ such that $\|f-fh\|_{A(G)}<\ve$ and 
     $\|f-\tilde g_\ve\|_{A(G)}<\ve/\|h\|_{A(G)}$. Setting $g_\ve=\tilde g_\ve h$, we find that  $g_\ve$ vanishes on $\nul(f)\subset\nul(h)\cup(\nul(f)\cap\supp{h})$ and satisfies
     $\|f-g_\ve\|_{A(G)}\leq\|f-fh\|_{A(G)}+\|fh-\tilde g_\ve h\|_{A(G)}\leq 2\ve$. This establishes the statement.  A similar result holds for the local version. If $S$ is compact the equivalence is clear without the additional assumption on $G$. 
   \end{remark}
   
   The next statement is a union result for strongly spectral sets, where we assume that either $G$ has the property that $u\in\overline{uA(G)}$, or $S_1$ and $S_2$ are compact.  The proof is similar to the one for Ditkin sets \cite{warner}.

%We find in this way 

\begin{theorem}\label{spunion}
Let $ S_1,S_2 $ be two strongly (local) spectral subsets of the locally compact 
group $ G $. Then the union $ S:=S_1\cup S_2 $ is also strongly (local) spectral. 

If $S_1$ and $S_2$ are closed subsets such that $S_1\cap S_2$ is strongly (local) spectral then $S_1\cup S_2$ is strongly (local) spectral if and only if so are $S_1$ and $S_2$. 
 \end{theorem}
\begin{proof}
We show the statement for strongly spectral sets. Let  $ u\in k({ S })$. Then $ u\in k({S_1}) $ and therefore for any $ \ve>0 $,
there exists $ u_1\in j(S_1)\cap k(\text{null}{(u)})$ such that $ \no{u-u_1}_{A(G)}<\ve $. 
 As $\text{null}{(u)}\supset S_2$, $u_1\in k(S_2)$ and, since $S_2$ is strongly 
spectral, 
there exists $u_2\in j(S_2)$ such that $\no{u_2-u_1}_{A(G)} <\ve $
and
 $ \text{null}{(u_2)}\supset \text{null}{(u_1)}$.  We   then have that 
$u_2 $ vanishes 
on $W_1\cup W_2$, where $W_1$ and $W_2$ are some neighborhoods of $S_1$ and 
$S_2$ respectively and hence $u_2\in j(S_1\cup S_2)=j(S)$.

Furthermore,
$$ \noag{u_2-u}\leq \noag{u_2-u_1}+\noag{u-u_1}<2\ve.$$
This shows that $ S $ is strongly spectral.

Assume now that $S_1$ and $S_2$ are closed subsets such that $S_1\cap S_2$ is strongly spectral. Suppose that $S_1\cup S_2$ is strongly spectral and let $u\in k(S_1)$ and $\ve>0$.  Then $u\in k(S_1\cap S_2)$ and there exists $g\in j(S_1\cap S_2)\cap k(\nul(u))$ such that $\|u-g\|_{A(G)}<\ve$. Let $C=S_2\cap\supp{g}$. It is disjoint from $S_1$ and hence there exists $w\in A(G)$ such that $w=1$ on $C$ and $w=0$ on a neighborhood of $S_1$. Set $h=g-gw$. Then $h$ vanishes on $C$ and hence on $S_2\subset (S_2\cap\supp{g})\cup\nul(g)$. As $g$ vanishes on $\nul(u)\supset S_1$, we obtain that $h\in k(S_1\cup S_2)$. Therefore there exists $g'\in j(S_1\cup S_2)\cap k(\nul(h))$ such that $\|h-g'\|_{A(G)}<\ve$. We have that $g'+wg\in j(S_1)\cap k(\nul(u))$ and 
$$\|u-g'-wg\|_{A(G)}\leq \|u-g\|_{A(G)}+\|g-wg-g'\|_{A(G)}<2\ve,$$
that is $S_1$ is strongly spectral. The proof for $S_2$ follows by symmetry. 
\end{proof}

Introducing the class of strongly spectral sets, our hope was that, using the technique developed in the previous section, we could prove that any set of local spectral synthesis is strongly spectral. 
Let 
 $T $ be a 
compact subset of an abelian, locally compact, $\sigma$-compact group $G $. 
Suppose that the function $u= d\ast c$ (in its canonical form) also  vanishes  on $T $, that is,
$d\in ((S\cup T)\cdot  c)^\perp  $. Let
  \begin{eqnarray*}
   r_k&:=&P_{(U_kS\cup T)\cdot c}(d), k\in\N,
 \end{eqnarray*}
 where $(U_k)_k$ is a sequence of decreasing neighborhoods of $e$ such that  $\limk P_{U_kS\cdot c}(d)=0 $.
Since the sequence of compact subsets   $(U_k)_{k\in\N} $
is decreasing, the sequence of closed subspaces  
$((U_kS\cup T)\cdot c )_{k\in\N} $ is also decreasing. Consequently, the limit
\begin{eqnarray*}  
 \nn r_\iy\nn  
&:= &
\limk r_k
\end{eqnarray*}
exists in $L^2(G)$.
The statement would follow if one could show that $\nn r_\iy=0$.

\section{Appendix}

%\subsection{Actions of $VN(G)$ on $A(G)$}
In this section we shall collect some of the properties of  $A(G)$ as an 
$VN(G)$-module.

\par
Following \cite[III.2]{ta} we  define the left and
right action of a von Neumann algebra $\mathcal M$ on its predual space 
${\mathcal 
M_*}$
 in the following way:
\begin{equation*}
\sp S{T\cdot u}:=\sp{ST} u,\ \ \sp S{u\cdot T}:=\sp {TS}u,\ S,T\in \mathcal M,\
 u\in {\mathcal M}_*.
  \end{equation*}

  Let us write $u=\ga_{f,g}\in A(G)$  for the coefficient of
  the left regular representation
  defined by  $f,g\in L^2(G)$, i.e. $\ga_{f,g}=\langle\lambda(t)f,g\rangle$.
  %and $\ga_{f,g}^r\in A(G)$ for the
  %coefficients  of the right regular representation, i.e
  %$$\ga_{f,g}^r(t)=\langle r(t)f,g\rangle.$$
  %Then
   %\begin{equation*}\label{}
 %\sp{r(t)V(f)} {V(g)}=:\ga_{V(f),V(g)}^r=\ga^\la_{f,g}(t):=
 %\sp {\la(t)f} g,\ t\in G.
  %\end{equation*}

%For $\va\in\{\lambda,r\}$ in what follows we denote by $VN_\va(G)$ the algebra 
%$VN(G)$ if $\va=\lambda$ and the algebra $VN_r(G)$ if $\va=r$.
From the definition of the right and left action of
$VN(G)$ on $A(G)$ it follows that
 \begin{equation}\label{leri}
 T \cdot \ga_{f,g}=\ga_{T (f),g},\ \ga_{f,g}\cdot T=\ga_{f,T^*(g)},\ f,g\in
 L^2(G),
 \end{equation}

 since for $T\in VN(G)$ and  $u=\ga_{f,g} \in A(G)$,
 $ \ f,g\in L^2(G)$,
 \begin{eqnarray}
 \nonumber \sp S {T\cdot  u} &=& \langle ST,u\rangle =\langle ST(f), g\rangle \\
 \nonumber  &=& \langle S(T(f)), g\rangle =\sp S{\ga_{T(f),g}},\\
  \nonumber\sp S {u\cdot T } &=& \sp{T S}u=\langle T S(f), g\rangle \\
 \nonumber  &=& \langle S(f), T^*(g)\rangle =\sp S{\ga_{f,T^*(g)}}.
 \end{eqnarray}
 %It is also easy to check that
 %\begin{equation}\label{corl}
  %\ga^r_{f,g} T=\ga^\la_{V(f),V(g)}\Phi(T),\ T\in VN_r(G),f,g\in L^2(G).
  %\end{equation}
Let also
\begin{equation*}
 \check u(t)=u(t\inv ),\  u\in A(G), t\in
 G.
 \end{equation*}
 We define an antilinear map $T\mapsto \ol T$ of $VN(G)$ by
 \begin{equation*}
 \ol T (f):= \ol{T(\ol f)},\ f\in L^2(G).
 \end{equation*}
and a linear  involution $T\mapsto \check T$ on $VN(G)$ by
 \begin{equation*}
 \sp {\check T} u:=\sp T{\check u},\ T\in VN(G),\ u\in A(G).
 \end{equation*}
 For $u=\ga_{f,g}$  we have
 \begin{eqnarray}
 \nonumber \check u(t)&=& u(t\inv)
  = \langle \la(t\inv) f, g\rangle
  =\ol {\langle \la(t) g, f\rangle}\\
  \nonumber &=&\langle \la(t)\ol g,\ol f\rangle =\ga_{\ol g, \ol f}(t),
 \end{eqnarray}
that is  \begin{equation}\label{check}
 \check\ga_{f,g}=\ga_{\ol g,\ol f},\ f,g\in L^2(G).
 \end{equation}
   
 Similarly
 \begin{equation}\label{tilde}
  \tilde\ga_{f,g}=\ga_{g,f},\ f,g\in L^2(G).
 \end{equation} 

Hence, for $u=\ga_{f,g}\in A(G)$ we see that
\begin{eqnarray}
 \nonumber \sp{\check T} u = \sp{T}{\check u}
   &=& \sp{T}{ 
    \ga_{\ol g, \ol f}
              }=\langle T(\ol g),\ol f\rangle\\
    \nonumber  &=&\langle \ol g,T^*(\ol f)\rangle=
    \langle\ol g,
    \ol{({\ol{T}^*(
    f)}}
\rangle=\langle\ol T^*(f),g\rangle,
\end{eqnarray}
whence  
 \begin{equation}\label{tista}
  \check T=\ol T^*,\ T\in VN(G).
  \end{equation}\

We have also the following identities
\begin{equation}\label{inversion}
 (u \cdot T)\tilde{}= T^*\cdot \tilde u,\quad (T \cdot u)\tilde{}=\tilde u\cdot T^*,
 \ u\in A(G),\ T\in VN(G).
 \end{equation}

Indeed, for  $u=\ga_{f,g}\in A(G)$, we have by  
(\ref{leri}) 
and (\ref{tilde})
$$(u\cdot  T)\tilde{}=\tilde\ga_{f,T^*(g)}=\ga_{T^*(g),f}=T^*\cdot \ga_{g,f}=T^*\cdot \tilde u.$$
 %\begin{eqnarray}\nonumber \sp S{(u T)\tilde{}} &=& \sp{ \tilde S}{u T}
 %= \sp{ (T \tilde S\ } u\\
  %\nonumber  &=& \sp{ (T\tilde S} {\ga^\va_{f,g}}=\sp{ \tilde S} 
%{\ga^\va_{f,T^*(g)}}=\sp{ S} {\ga^\va_{T^*(g),f}}
 % \\
%\nonumber &=& \sp {S((T^*(g))}f=\sp S{T^* \tilde u}.
%\end{eqnarray}
The other equality is obtained in a similar way.                        

In \cite{E} Eymard considered another action
$(T,u)\mapsto T\circ u$ of
 $VN(G)$ on $A(G)$ which is defined through the following formula:
 \begin{equation*}
 \sp{S}{T\circ u}:= \sp {\check T S} u,\ u\in A(G).
 \end{equation*}
 It follows from the relations above that
\begin{equation}\label{eyac}
 T\circ (\ga_{f,g})=\ga_{f,\ol T(g)}.
  \end{equation} 
  Indeed, for $u=\ga_{f,g}\in A(G)$ we get
 \begin{eqnarray}
 \nonumber \sp{S}{T\circ u} &=& \sp{\check  T S}u
= \langle\check T S (f), g\rangle \\
 \nonumber &=&\langle S(f),\ol T (g)\rangle=\sp S{\ga_{f,\ol T(g)}}
 \end{eqnarray}

 Hence, by (\ref{leri}) and (\ref{tista})
 \begin{eqnarray}\label{leey}
T\circ u=u\cdot  \check T, \ u\in A(G),\ T\in VN(G).
 \end{eqnarray}
 In particular,  it follows from \cite[Proposition 3.17]{E} that for $u\in 
A(G)\cap L^2(G)$ and $T\in VN(G)$, $T(u)\in A(G)\cap L^2(G)$,  and
 \begin{equation}\label{l2ag}
 T(u)=T\circ u=u \cdot \check T.
 \end{equation}

{\bf Acknowledgments.}The work was partially written when the first author was visiting
Chalmers University of Technology in G\"oteborg, Sweden and when
the second author was a visiting professor at Universit\'e de Lorraine, 
France, whose hospitality is highly acknowledged.   The authors would like to thank Victor Shulman for 
valuable  discussions and remarks.

 \end{document}